\documentclass[letterpaper]{amsart}
\pdfoutput=1
\usepackage{amsmath}
\usepackage{amsfonts}
\usepackage{amssymb}
\usepackage{amsthm}
\usepackage{graphicx}
\usepackage{bbm}
\usepackage{hyperref}
\usepackage{cite}

\def\thetitle{Fiberwise homogeneous fibrations of the 3-dimensional space forms by geodesics}
\def\theauthor{Haggai Nuchi}

\hypersetup{pdftitle=\thetitle, pdfauthor=\theauthor, colorlinks=true, linkcolor=black, citecolor=black, filecolor=black, urlcolor=black}

\newcommand{\R}{\mathbb{R}}
\newcommand{\F}{\mathcal{F}}
\newcommand{\C}{\mathbb{C}}
\newcommand{\Isom}{\operatorname{Isom}}

\renewcommand{\Re}{\operatorname{Re}}
\renewcommand{\Im}{\operatorname{Im}}

\newtheorem{thm}{Theorem}[section]

\newtheorem{lem}[thm]{Lemma}

\theoremstyle{definition}
\newtheorem{dfn}[thm]{Definition}

\theoremstyle{remark}
\newtheorem*{rmk}{Remark}

\title[\small \thetitle]{\LARGE \thetitle}
\author{\theauthor}
\date{}

\begin{document}
  \begin{abstract}
    A fibration of a Riemannian manifold is {\em fiberwise homogeneous} if there are isometries of the manifold onto itself, taking any given fiber to any other one, and preserving fibers.

    Examples are fibrations of Euclidean $n$-space by parallel $k$-planes, and the Hopf fibrations of the round $n$-sphere by great $k$-spheres.

    In this paper, we describe all the fiberwise homogeneous fibrations of Euclidean and hyperbolic 3-space by geodesics. Our main result is that, up to fiber-preserving isometries, there is precisely a one-parameter family of such fibrations of Euclidean 3-space, and a two-parameter family in hyperbolic \mbox{3-space}.

    By contrast, we show in another paper that the only fiberwise homogeneous fibrations of the round 3-sphere by geodesics (great circles) are the Hopf fibrations.

    More generally, we show in that same paper that the Hopf fibrations in all dimensions are characterized, among fibrations of round spheres by smooth subspheres, by their fiberwise homogeneity.
  \end{abstract}
  \maketitle
  
  \section{Introduction}
    \subsection{Background}
      The Hopf fibrations of $S^{2n+1}$ by great circles, $S^{4n+3}$ by great $3$-spheres, and $S^{15}$ by great $7$-spheres have a number of interesting properties. For one, their fibers are parallel, and moreover, the fibrations are characterized by this property \cite{gromoll1988low, wilking2001index}.
    
      The Hopf fibrations are also {\em fiberwise homogeneous}.
      \begin{dfn}
        Let $\F$ be a fibration of a riemannian manifold $(M,g)$. We say that $\F$ is {\em fiberwise homogeneous} if for any two fibers there is an isometry of $(M,g)$ taking fibers to fibers and taking the first given fiber to the second given fiber.
      \end{dfn}
    
      In another paper \cite{nuchi2014hopf}, we prove that the Hopf fibrations are characterized by being fiberwise homogeneous among all fibrations of spheres by smooth subspheres.
      
      But in order for us to really {\em feel} what ``fiberwise homogeneous'' means, we need more examples. We don't gain any insight into this property just from knowing that the Hopf fibrations are the only such fibrations of spheres by subspheres. The point of this paper is to provide examples. See also \cite{nuchi2014surprising} for an example of a fiberwise homogeneous fibration of the Clifford torus $S^3\times S^3$ in the 7-sphere by great \mbox{3-spheres} which is not part of a Hopf fibration.
      
      The main theorem of this paper is the following:
      \begin{thm}\label{Thm:SpaceFormFWHs}
        The fiberwise homogeneous fibrations of Euclidean 3-space by lines form a 1-parameter family. The fiberwise homogeneous fibrations of Hyperbolic 3-space by geodesics form a 2-parameter family, plus one additional fibration.
      \end{thm}
      \begin{rmk}
        As part of the proof of Theorem~\ref{Thm:SpaceFormFWHs}, we construct the fibrations explicitly. See Figure~\ref{Fig:Examples} for pictures of typical examples. Technically, we classify fibrations up to ``equivalence''; we don't consider e.g.~a fibration by parallel horizontal lines as different from a fibration by parallel vertical lines. There are many fibrations of Euclidean and hyperbolic spaces by geodesics; see for instance two papers by Salvai \cite{salvai2009global} and Godoy--Salvai \cite{godoysalvai2014}.
      \end{rmk}
      \begin{figure}[h]
        \centering
        \includegraphics[width=0.35\textwidth]{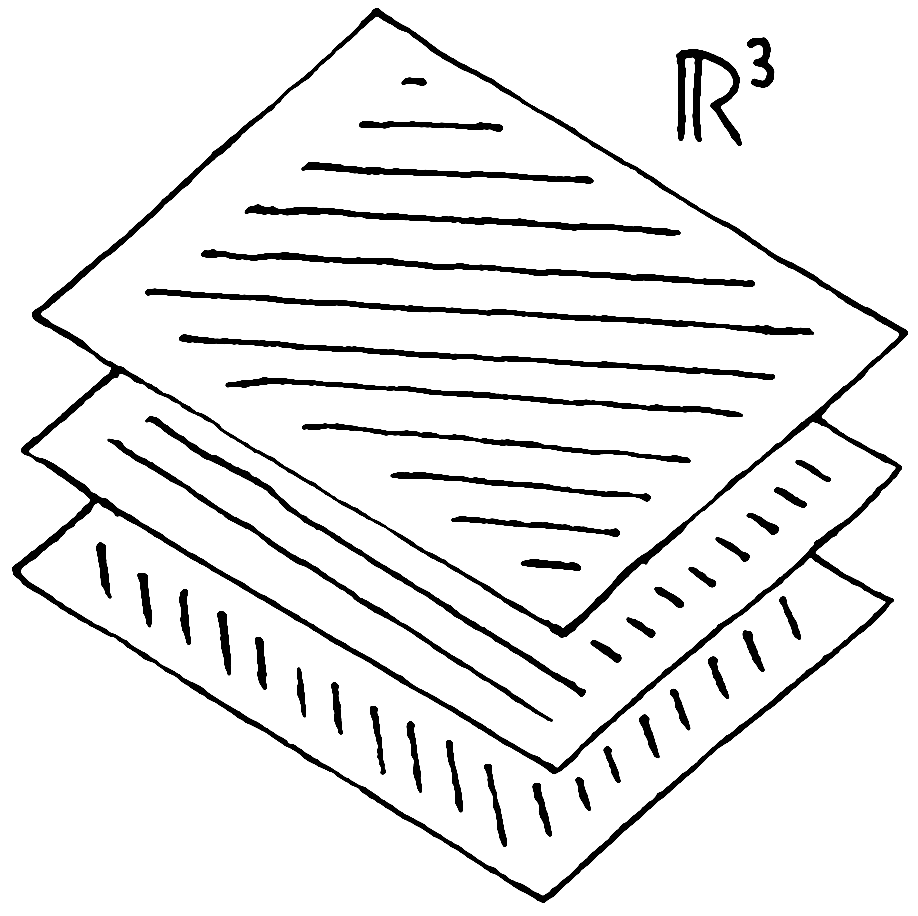}~        \includegraphics[width=0.55\textwidth]{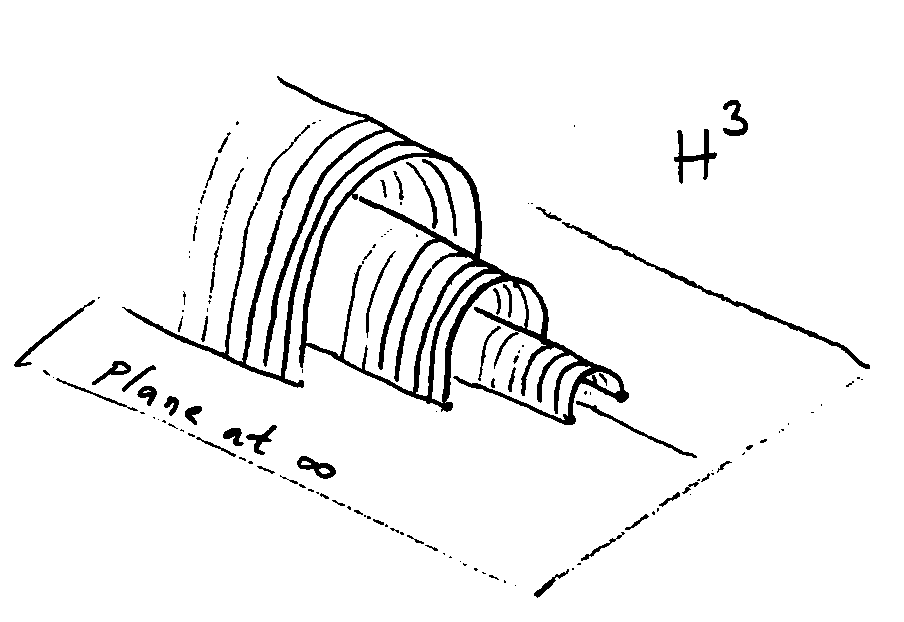}
        \caption{Typical fiberwise homogeneous fibrations of $\R^3$, left, and $H^3$, right. In the pictured fibration of $\R^3$, we layer Euclidean space by parallel planes, and fiber each plane by parallel lines whose direction changes at a constant rate as we move from plane to plane. In the pictured fibration of $H^3$, geodesics form nested tunnels over a line in the plane at infinity in the upper half-space model of $H^3$.}
        \label{Fig:Examples}
      \end{figure}
      
      Here is a rough outline of the classification of the fiberwise homogeneous fibrations of Euclidean and Hyperbolic 3-space by geodesics. Let $X$ denote either $E^3$ or $H^3$ (the proof has a similar form in both cases). Start with a hypothetical fiberwise homogeneous fibration $\F$ of $X$ by geodesics. Let $G$ be a subgroup of $\Isom(X)$ which acts transitively on the fibers of $\F$. For technical reasons, we may assume without loss of generality that $G$ is connected and closed. The conjugacy classes of closed connected subgroups of $\Isom(X)$ are well understood in the literature. By letting $G$ range across all closed connected subgroups of $\Isom(X)$, we can analyze the geometry of $G$ to discover which fiberwise homogeneous fibrations by geodesics it could possibly act on. For example, 1-parameter groups of isometries are too small to act transitively on $\F$, and the full isometry group is too large to preserve $\F$. Like Goldilocks in her interactions with the three bears, we check each intermediate group of isometries to see if it is ``just right.'' In this way we discover all fiberwise homogeneous fibrations by geodesics. Finally, after eliminating redundancies in the list (some may be equivalent to one another), we arrive at the list of fiberwise homogeneous fibrations of $X$ by geodesics, up to equivalence.
      
      That concludes the outline of the proof. In Section~\ref{Sec:Technical}, we have some technical preliminaries. In Sections~\ref{Sec:Euc} and~\ref{Sec:Hyp}, we describe the fiberwise homogeneous fibrations of Euclidean and Hyperbolic 3-space, respectively.

      The interested reader can consult the bibliography of \cite{nuchi2014hopf} for a more detailed list of related papers.
      
    \subsection{Acknowledgments}
      This article is an extended version of a portion of my doctoral dissertation. I am very grateful to my advisor Herman Gluck, without whose encouragement, suggestions, and just overall support and cameraderie, this would never have been written.
      
      Thanks as well to the Math Department at the University of Pennsylvania for their support during my time there as a graduate student.
    
  \section{Technical Preliminaries}\label{Sec:Technical}
    The following definition and lemmas will be useful to us in both of the later sections, so we gather them here.
    \begin{dfn}
      Let $\F_1$ and $\F_2$ be two fibrations of a Riemannian manifold $(M,g)$. We say that $\F_1$ and $\F_2$ are {\em equivalent} if $\exists T\in \Isom(M,g)$ such that $\F_2=T(\F_1)$.
    \end{dfn}
    \begin{lem}\label{Lem:equivalent}
      If $\F_1$ is equivalent to $\F_2$ and $\F_1$ is fiberwise homogeneous, then $\F_2$ is also fiberwise homogeneous. Moreover, the groups of isometries acting on $\F_1$ and $\F_2$ are conjugate to one another in the full isometry group of $(M,g)$.
    \end{lem}
    \begin{proof}
      If $G \subset \Isom(M,g)$ acts isometrically and transitively on the fibration $\F_1$, and $\F_2=T(\F_1)$, then $TGT^{-1}$ acts isometrically and transitively on $\F_2$.
    \end{proof}
    The following lemma will be quite useful to us. If we have a fiberwise homogeneous fibration $\F$ of a Riemannian manifold $M$, it will allow us to assume that a group $G$ acting transitively on $\F$ is a connected Lie subgroup of $\Isom(M)$.
    \begin{lem}\label{Lem:ClosedConnected}
      Let $M$ be a connected Riemannian manifold, and let $G$ be a subgroup of $\Isom(M)$. Denote by $\overline{G}_0$ the identity component of the closure of $G$. Suppose $G$ acts transitively on a smooth fibration $\F$ of $M$. Then, $\overline{G}_0$ acts transitively on $\F$ as well.
    \end{lem}
    \begin{proof}
      Let $G$ preserve the fibers of $\F$. Let $\{g_n\}_{n=1}^{\infty} \subset G \subseteq \Isom(M)$, and let $g_n \to g\in \Isom(M)$. If each $g_n$ preserves each fiber of $\F$, then the limit $g$ clearly does as well. Thus $\overline{G}$ preserves the fibers of $\F$. Also $G\subseteq \overline{G}$, so $\overline{G}$ acts transitively on $\F$ as well.
      
      Now let $G$ be a closed disconnected Lie subgroup of $\Isom(M)$ acting transitively on $\F$. Let the manifold $B$ be the base of $\F$. Then $B$ has the structure of a connected homogeneous space, and hence is diffeomorphic to $G/H$, where $H$ is the isotropy subgroup of $G$ fixing a point. Let $G_i$ be the connected components of $G$, $G_0$ the identity component. The subgroup $H$ intersects every $G_i$, or else $B$ would be disconnected, so there are $g_i\in H\cap G_i$ for all $i$. Then $g_i H = H$, from which it follows that the image of $G_0$ intersects the image of every $G_i$ in $G/H$. But as $g_0$ ranges across all elements of $G_0$, $g_0 g_i$ ranges across all elements of $G_i$, so $g_0 g_i H = g_0 H$ and the image of $G_0$ is identical to the image of each $G_i$. Thus the image of $G_0$ is all of $G/H$. Therefore the identity component of $G$ acts transitively on $B$ and hence on $\F$.
    \end{proof}

  \section{Euclidean 3-space}\label{Sec:Euc}
    In order to carry out the classification which was just briefly described, we need a description of the conjugacy classes of closed connected subgroups of the isometry group of $E^3$. A full list can be found in \cite{beckers2008subgroups}. We reproduce the list here, and describe each element in detail. After describing these groups, we will give the list of fiberwise homogeneous fibrations of Euclidean 3-space, followed by the proof that our list is exhaustive.
    
    In the following list, each item is written as ``Name of group, dimension of group. Description of group.'' For the name of the group, we either choose a name from \cite{beckers2008subgroups}, or invent our own; we try to give descriptive names where possible. In the description of the group, we choose a concrete representative of the conjugacy class, acting on the usual coordinates $(x,y,z)$ of $E^3$. We remark that \cite{beckers2008subgroups} actually describes the subalgebras of the Lie algebra of $\Isom(E^3)$, but these are in one-to-one correspondence with closed connected subgroups of $\Isom(E^3)$ via exponentiation.
    
    Closed connected groups of isometries of $E^3$:
    \begin{itemize}
      \item $\{1\},\dim 0.$ The trivial group.
      \item $T(1),\dim 1.$ The group of translations along the $z$-axis.
      \item $SO(2),\dim 1.$ The group of rotations around the $z$-axis.
      \item $\overline{SO(2)}_t,\dim 1.$ The group of screw-translations along the $z$-axis with pitch $t$. Distinct $t\neq 0$ give distinct conjugacy classes. On the Lie algebra level, this group is generated by $-t \frac{\partial}{\partial x} + t \frac{\partial}{\partial y} + \frac{\partial}{\partial z}$. The notation is meant to remind us that this group is the universal covering group of $SO(2)$. These are distinct for different $t$, except that $\overline{SO(2)}_{t}$ is conjugate to $\overline{SO(2)}_{-t}$ by an orientation-reversing isometry of $E^3$.
      \item $SO(2)\times T(1),\dim 2.$ The group of rotations around and translations along the $z$-axis.
      \item $T(2), \dim 2.$ The group of translations in the $xy$-plane.
      \item $\overline{E(2)}_t, \dim 3.$ A universal covering group of the isometries of the Euclidean plane. It consists of translations in the $xy$-plane together with screw-translations along the $z$-axis with pitch $t$. It's generated by $T(2)$ and $\overline{SO(2)}_{t}$. Just as with $\overline{SO(2)}_{t}$, these are distinct for different $t$, except that $\overline{E(2)}_t$ is conjugate to $\overline{E(2)}_{-t}$ by an orientation-reversing isometry of $E^3$.
      \item $T(3), \dim 3.$ The group of translations of Euclidean 3-space.
      \item $E(2), \dim 3.$ The group of translations and rotations of the $xy$-plane.
      \item $SO(3), \dim 3.$ The group of rotations around the origin $(0,0,0)$.
      \item $E(2)\times T(1), \dim 4.$ The group generated by translations and rotations of the $xy$-plane, and translations in the $z$-axis.
      \item $E(3), \dim 6.$ The identity component of the full isometry group of $E^3$. Generated by translations and rotations.
    \end{itemize}
    
    Now we describe all fiberwise homogeneous fibrations of $E^3$ by geodesics. Let $\F_t$ (See Figure~\ref{Fig:EucFib}) be the fibration consisting of the integral curves of the vector field
      \[ v_t = \cos(tz)\frac{\partial}{\partial x} + \sin(tz)\frac{\partial}{\partial y}. \]
    In other words, layer $E^3$ by parallel planes. Fiber each plane by parallel straight lines, but vary the angle of the lines at a constant rate as we move from plane to plane. The variable $t$ controls the rate of change of the angle.
    \begin{figure}[h]
      \centering
      \includegraphics[width=0.5\textwidth]{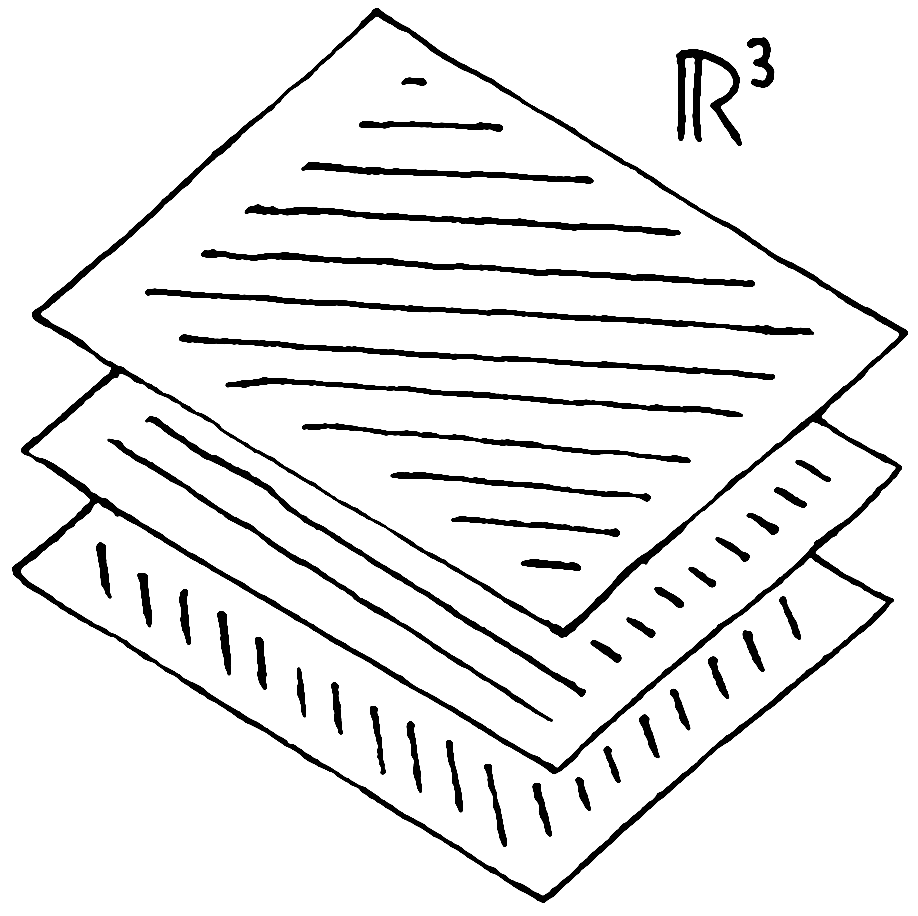}
      \caption{A fiberwise homogeneous fibration of $\R^3$.}
      \label{Fig:EucFib}
    \end{figure}
    
    \begin{thm}\label{Thm:EuclideanFibrations}
      The fiberwise homogeneous fibrations of Euclidean 3-space by lines, up to equivalence, are precisely $\F_t$ for each $t\in\R, t\geq 0$, and no others.
    \end{thm}
    \begin{proof}
      First we will show that each $\F_t$ is fiberwise homogeneous. Then we will show that no other fibrations are fiberwise homogeneous. Then we will show that no $\F_t$ is equivalent to $\F_{t'}$ for $t,t'\geq 0$ and $t\neq t'$.
    
      First note that $T(3)$ (for example) acts transitively on $\F_0$, and that $\overline{E(2)}_t$ acts transitively on $\F_t$ for $t\neq 0$, so that these fibrations really are fiberwise homogeneous.
      
      Now, let $\F$ be a fiberwise homogeneous fibration $\F$ by geodesics, and let $G$ be a group of isometries of $E^3$ acting transitively on $\F$. Without loss of generality, by Lemma~\ref{Lem:ClosedConnected} we may assume that $G$ is closed and connected. By replacing $\F$ with an equivalent fibration if necessary (Lemma~\ref{Lem:equivalent}), we may also assume that $G$ is exactly the representative of its conjugacy class which is listed in the table above. We will now show that $\F$ must be $\F_t$ for some $t$.
      \begin{itemize}
        \item Suppose $G=\{1\}$, $T(1)$, $SO(2)$, or $\overline{SO(2)}_t$. Then the dimension of the group is too small to act transitively on $\F$, because the image of a single fiber under $G$ will be either 1- or 2-dimensional. This is a contradiction.
        \item Suppose $G=SO(2)\times T(1)$. Consider the fiber $F$ through the origin. If $F$ is the $z$-axis, then $G$ fixes it and hence cannot act transitively on $\F$. If $F$ is not the $z$-axis, then a rotation in $G$ which fixes the origin will take $F$ to a different line through the origin, and hence $G$ does not preserve $\F$. This is a contradiction.
        \item Suppose $G=T(2)$, and consider the fiber $F$ through the origin. If $F$ lies in the $xy$-plane, then $G$ cannot take $F$ out of the $xy$-plane and hence does not act transitively on $\F$. If $F$ is transverse to the $xy$-plane, then the image of $F$ under $G$ is a fibration by parallel straight lines and hence is equivalent to $\F_0$.
        \item Suppose $G=\overline{E(2)}_t$. Consider the fiber $F$ through the origin. If $F$ is transverse to the $xy$-plane, then a screw-translation around the origin by an angle of $\pi$ will take $F$ to a fiber which intersects $F$ transversely, and hence $G$ does not preserve $\F$, a contradiction. If $F$ lies in the $xy$-plane, then the image of $F$ under $G$ is equivalent to the fibration $\F_t$: translations move it to its parallel translates in the $xy$-plane, and screw-translations move the $xy$-plane in the $z$-direction while controlling the angle of the lines. 
        \item Suppose $G=T(3)$. Let $F$ be any fiber of $\F$. The image of $F$ under $G$ is equivalent to $\F_0$, so $\F=\F_0$.
        \item Suppose $G=E(2)$ or $G=E(2)\times T(1)$. Consider the fiber $F$ through the origin. If $F$ is not the $z$-axis, then a rotation of the $xy$-plane about the origin takes $F$ to a line which intersects $F$ transversely, a contradiction. Thus $F$ is the $z$-axis, and the image of $F$ under $G$ is equivalent to $\F_0$.
        \item Suppose $G=SO(3)$ or $G=E(3)$. Consider the fiber $F$ through the origin. Apply a rotation about the origin which does not fix $F$. Then the image of $F$ is a line which intersects $F$ transversely, and so $G$ does not preserve $\F$, a contradiction.
      \end{itemize}
      Thus we conclude that $\F$ is equivalent to $\F_t$ for some $t$.
      
      Finally we show that no two such fibrations are equivalent. Observe that the unit vector field $v_t$ along $\F_t$ is an eigenfield for the curl operator on $E^3$ with eigenvalue $t$. Curl eigenfields remain curl eigenfields after an isometry and the eigenvalue is preserved up to sign (depending on the orientation of the isometry), so no two such fibrations are equivalent.
    \end{proof}
    
    \section{Hyperbolic 3-space}\label{Sec:Hyp}
    We classify fiberwise homogeneous fibrations of Hyperbolic 3-space by geodesics by the same method as for Euclidean 3-space. We start with a description of the closed connected groups of isometries of Hyperbolic 3-space. The list below can be found in \cite{shaw1970subgroup}. Just as with the analogous list for Euclidean 3-space, the list in \cite{shaw1970subgroup} is actually of Lie subalgebras of the Lie algebra corresponding to the isometry group of Hyperbolic 3-space. To be precise, the theorem actually discusses the Lie algebra for the group $SO(1,3)$, but this group is isomorphic to the isometry group of $H^3$.
    
    In the following list, each item is once again written as ``Name of group, dimension of group. Description of group.'' In the description of the group, we choose a concrete representative of the conjugacy class. We will use some basic facts about the geometry of Hyperbolic 3-space; see Thurston \cite{thurston1997three} for more information. Sometimes it will be useful for us to use the upper half-space model of Hyperbolic 3-space, which we give the coordinates $(z,x)\in \C\times \R^+$, and we denote points on the plane at infinity as $(z,0)$. Other times we may use the Poincar\'e disk model, which we give the coordinates $v=(x,y,z)\in\R^3$ with $\|v\|<1$, and with the sphere at infinity consisting of those $v$ with $\|v\|=1$.
    
    We remind the reader that geodesics in either of those models consist of straight lines or arcs of circles, both of whose endpoints meet the plane at infinity or sphere at infinity orthogonally. In the upper half-space model, that includes vertical lines (i.e.~sets of the form $\{z\}\times\R^+$). We also take advantage of the extremely useful fact that orientation-preserving isometries of Hyperbolic 3-space extend to M\"obius transformations of the sphere or plane at infinity (see \cite{thurston1997three}); in what follows we will describe a group of hyperbolic isometries by its effect on the plane or sphere at infinity.
    
    Closed connected groups of isometries of $H^3$:
    \begin{itemize}
      \item $\{1\},\dim 0.$ The identity.
      \item $Hyp,\dim 1.$ The group of dilations of the plane at infinity in the upper half-space model. These are known as hyperbolic transformations.
      \item $Par,\dim 1.$ The group of translations of the plane at infinity in the upper half-space model by real directions. These are known as parabolic transformations.
      \item $Ell,\dim 1.$ The group of rotations of the plane at infinity about the origin. These are known as elliptic transformations. 
      \item $Lox,\dim 1.$ The group of screw-dilations of the plane at infinity about the origin. The infinitesimal generator of this group is a linear combination of the infinitesimal generators for the hyperbolic and elliptic groups $H(1)$ and $E(1)$. There are actually an infinite number of conjugacy classes of groups of this type, corresponding to different ratios of this linear combination. ``$Lox$'' is for loxodromic.
      \item $T(2),\dim 2.$ The group of translations of the plane at infinity.
      \item $\langle Hyp,Par \rangle,\dim 2.$ The group generated by the hyperbolic transformations and parabolic transformations: translations of the plane at infinity in real directions and dilations about points on the real axis.
      \item $\langle Ell,Hyp \rangle,\dim 2.$ The group of dilations and rotations about the origin in the plane at infinity.
      \item $Hom,\dim 3.$ The homothety group of the plane at infinity, consisting of translations and dilations.
      \item $ScrewHom,\dim 3.$ The screw-homothety group of the plane at infinity, generated by translations as well as a loxodromic group (of type $Lox$ above). Just as with the loxodromic groups, there are infinitely many nonequivalent groups of this type.
      \item $E(2),\dim 3.$ The group of Euclidean transformations of the plane at infinity. Generated by translations and rotations.
      \item $H(2),\dim 3.$ The identity componenet of the group of isometries of a totally geodesic Hyperbolic plane. Concretely, we choose the hyperbolic 2-plane in the Poincar\'e disk model with $z=0$. This group happens to be isomorphic to $PSL(2,\R)$, though we will not need this fact.
      \item $SO(3),\dim 3.$ The group of orthogonal transformations of the sphere at infinity in the Poincar\'e disk model.
      \item $Sim,\dim 4.$ The group of similitudes of the plane at infinity. Generated by translations, rotations, and dilations.
      \item $H(3),\dim 6.$ The identity component of the full isometry group of Hyperbolic 3-space.
    \end{itemize}
    That's the full list of conjugacy classes of connected closed subgroups of the isometry group of $H^3$.
    
    We now describe the fiberwise homogeneous fibrations of Hyperbolic 3-space by geodesics. We use the upper half-space model. To start, there's one special fibration which is qualitatively different from the others. Let $\F_{\infty}$ be the fibration consisting of vertical lines, i.e.~lines of the form $\{z\}\times\R^+$. We will see that the remaining fibrations are similar to one another. Recall that we have identified the plane at infinity with $\C\times\{0\}$. Let's ignore the second coordinate and just identify the plane at infinity with $\C$. Define the set $S\subset \C$ by
      \[ S = \{z\in\C\ : \Im z \geq 1, \Re z \geq 0 \}. \]
    That is, $S$ is the first quadrant in the complex plane, shifted one unit in the positive imaginary direction. For each $z\in S$, we define an unparametrized geodesic $\gamma_z$ in Hyperbolic 3-space by assigning its endpoints in $\C$ to be $-i$ and $z$. Now let the group $\langle Hyp,Par \rangle$ act on $\gamma_z$. The image is all of $H^3$ and thus forms a fibration of $H^3$, which we name $\F_z$. To see this, note that $Par$ translates $\gamma_z$ in real directions, so the image of $\gamma_z$ under $Par$ is an infinitely long tunnel over the real line in the plane at infinity. The image of this tunnel under $Hyp$ is all of $H^3$, because $Hyp$ consists of dilations of the plane at infinity about points on the real line. Thus $\langle Hyp,Par \rangle$ acts transitively on the fibration $\F_z$. See Figure~\ref{Fig:Fz}.
    \begin{figure}[h]
      \centering
      \includegraphics[width=0.5\textwidth]{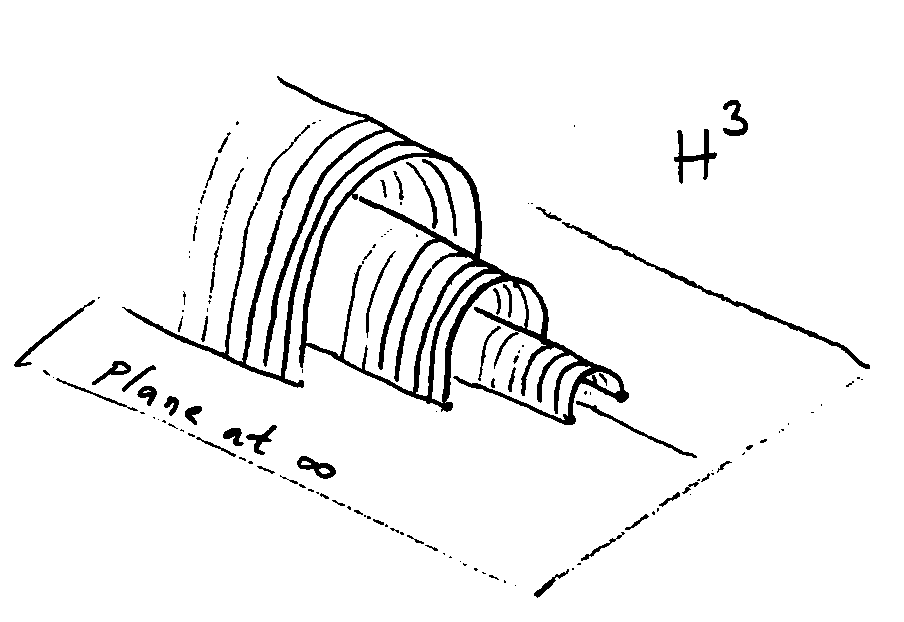}
      \caption{A typical picture of an $\F_z$ in the upper half-space model of $H^3$.}
      \label{Fig:Fz}
    \end{figure}
    
    We're technically not restricted to choosing $\gamma_z$ with $z$ in $S$. We could just as easily choose any $z$ with $\Im > 0$ and define $\gamma_z$ in the same way. But then we would have some redundancies in our list: some $\gamma_z$ would be equivalent to $\gamma_{z'}$ for $z \neq z'$.
    
    \begin{thm}\label{Thm:HyperbolicFibrations}
      The fiberwise homogeneous fibrations of Hyperbolic 3-space by geo-desics, up to equivalence, are precisely the fibrations $\F_z$ for $z\in S$, together with $\F_{\infty}$.
    \end{thm}
    \begin{proof}
      Just as with Euclidean 3-space, the proof of this theorem is in three parts. First, we demonstrate that the fibrations we list in the statement really are fiberwise homogeneous, then we prove that there are no others, and then we prove that our list has no redundancies.
      
      First, note that $\langle Hyp, Par \rangle$ acts transitively on each $\F_z$, by construction, so they are all fiberwise homogeneous. Note also that $Sim$ (the group of similitudes of the plane at infinity) acts transitively on $\F_{\infty}$, so it is fiberwise homogeneous as well.
      
      Now let $\F$ be a fiberwise homogeneous fibration of $H^3$ by geodesics, and let $G$ be a group of isometries acting transitively on $\F$. By Lemma~\ref{Lem:ClosedConnected}, we may assume $G$ is closed an connected. By replacing $\F$ by an equivalent fibration if necessary, we may assume that $G$ appears on the list above of closed connected groups of isometries of $H^3$.
      \begin{itemize}
        \item Suppose $G=\{1\}$, $Hyp$, $Par$, $Ell$, or $Lox$. As in the Euclidean case, the image of a fiber under $G$ will have dimension only $1$ or $2$ and hence cannot fill all of $H^3$. This is a contradiction.
        \item Suppose $G=T(2)$, $Hom$, $ScrewHom$, $E(2)$, or $Sim$. Let $F$ be any fiber in $\F$, in the upper half-space model. Suppose $F$ is a semicircle meeting the plane at infinity orthogonally. Then its endpoints form a line in the plane at infinity. Apply a translation to $F$ along this line. Then $F$ moves to another geodesic which intersects $F$ transversely, so $G$ does not preserve $\F$. This is a contradiction, so in fact $F$ must be a vertical line. Now the image of $F$ under $G$ is $\F_{\infty}$, so $\F=\F_{\infty}$.
        \item Suppose $G=\langle Hyp,Par \rangle$. Let $F$ be any fiber. Suppose $F$ is a vertical line in the upper half-space model. Then the image of $F$ under $G$ will fill out only points with positive imaginary part, or only points with negative imaginary part, or points with only zero imaginary part, depending on where the endpoint of $\F$ is. Thus $G$ does not act transitively on $\F$, a contradiction. Thus $F$ must be a semi-circle in the upper half-space model. As we just observed, $G$ preserves the set of points with positive (respectively negative, respectively zero) imaginary part, so $F$ must have one endpoint with positive imaginary part and one endpoint with negative imaginary part on the plane at infinity. Then the image of $F$ under $G$ fill out $H^3$ and hence $\F=\F_z$ for some $z\in \C$. If $z$ is not in $S$, then $\F$ is equivalent to $\F_{z'}$ for $z'\in S$: simply apply to $\F$ a rotation of $\pi$ about the origin and/or a reflection in the imaginary line in the plane at infinity.
        \item Suppose $G=\langle Ell,Hyp \rangle$. Let $F$ be the fiber of $\F$ in the upper half-space model passing through $(0,1)\in \C\times\R^+$. Suppose $F$ is a semicircle. Then a rotation applied to $F$ moves $F$ to a semicircle intersecting $F$ transversely. Thus $G$ does not preserve $\F$, a contradiction. Therefore $F$ must be a vertical line. But then $G$ fixes $F$, rather than acting transitively on $\F$, also a contradiction.
        \item Suppose $G=H(2)$. Let $F$ be the fiber through $0\in\R^3$ in the Poincar\'e disk model. The isometry group of the hyperbolic 2-plane consisting of points $(x,y,z)\in\R^3$ with $z=0$ contains the group of rotations about the $z$-axis. If $F$ is not the portion of the $z$-axis with norm less than 1, then these rotations take $F$ to a line segment which intersects $F$ transversely. Thus $F$ is a portion of the $z$-axis. Now the image of $F$ under $G$ defines a fibration of $H^3$ consisting of all the geodesics orthogonal to a fixed hyperbolic \mbox{2-plane}. Moving from the disk model to the upper half-space model, take this hyperbolic 2-plane to be the 2-plane over the real line in the plane at infinity. We see that $\F$ is precisely the fibration $\F_i$.
        \item Suppose $G=SO(3)$ or $H(3)$. Consider the fiber $F$ through the origin in the Poincar\'e disk model. Apply some rotation of the sphere at infinity which does not fix $F$; that moves $F$ to a line which intersects $F$ transversely, a contradiction.
      \end{itemize}
      Thus $\F$ is equivalent to some $\F_z$ for $z\in S$, or $\F$ is $\F_{\infty}$.
      
      We close out the chapter by showing that no two fiberwise homogeneous fibrations in our list are equivalent to one another. Whereas in the Euclidean case we made use of the fact that the unit vector fields along our fibrations were curl eigenfields with distinct eigenvalues, here we are unable to find so pretty an argument. What follows is rather technical and unenlightening. 
      
      Suppose that $\F_z$ is equivalent to $\F_{z'}$ for $z\neq z'$ and both $z$ and $z'$ in $S$. It's clear that $\F_{\infty}$ is not equivalent to $\F_z$ for any $z$, and that $\F_i$ is not equivalent to $\F_z$ for $z\neq i$; by the above discussion, they have different symmetry groups. Thus $G=\langle Hyp,Par \rangle$ is the identity component of the symmetry group of $\F_z$ and $\F_{z'}$.
      
      Let $T$ be an isometry of $H^3$ taking $\F_z$ to $\F_{z'}$. Then $TGT^{-1}(\F_{z'}) = \F_{z'}$, so $T$ is in the normalizer of $G$. If $\F_z$ equals $\F_{z'}$ as fibrations, we're done, so assume that $T$ is not in $G$. At the sphere of infinity in the disk model, $G$ fixes exactly one point (if we identify the sphere at infinity with the Riemann sphere $\C\cup \{\infty\}$, with $\C$ the plane at infinity in the half-space model, then $G$ fixes $\infty$). Therefore $T$ must fix that point as well. Thus $T$ acts as a similitude of the plane at infinity in the upper half-space model, possibly orientation-reversing.
      
      The group $G$ also preserves the real line (and no other affine line) in the plane at infinity, so $T$ must preserve this line as well. The similitudes of the plane which accomplish this are generated by $G$ and by the following transformations: rotations of the plane at infinity by $\pi$ about 0, reflections in the real axis, and reflections in the imaginary axis. The action of $G$ preserves $\F_z$, so we can focus on the effect of the last three transformations on $\F_z$.
      
      Reflection in the imaginary axis takes the geodesic joining $-i$ to $z$ to the geodesic joining $-i$ to $-\overline{z}$, so it takes $\F_z$ to $\F_{-\overline{z}}$. If $z$ and $-\overline{z}$ are both in $S$ then we must have $z=-\overline{z}=z'$, with $z$ on the imaginary axis.
      
      Consider the line segment joining $-i$ to $z$ in the plane at infinity. The real axis splits this line segment into two pieces, the ratio of whose sizes (top to bottom) is $\Im z$ to $1$.
      
      Consider the geodesic fiber $F$ in $\F_z$ which has one endpoint at $i$, and call the other endpoint $w$. The real axis also cuts the line segment joining $i$ and $w$, and the ratio of their lengths (top to bottom) is also $\Im z$ to $1$. Applying a rotation of the complex plane by $\pi$ about the origin to $\F_z$, or applying a reflection in the real axis, moves $F$ so that one endpoint is now at $-i$, and $w$ is taken to $z'$.
      
      Consider the line segment joining the endpoints of the geodesic joining $-i$ to $z'$. The ratio of the lengths (top to bottom) is now $1$ to $\Im z$, because we have flipped it upside down. Thus for a rotation by $\pi$ or a reflection about the real axis, we have $\Im z' = 1/\Im z$. If $z$ and $z'$ are both in $S$, then we must have $\Im z = \Im z' = 1$. We also must have just applied a rotation and not a reflection, because reflection takes $z$ in $S$ to $z'$ with nonpositive real part, and we assumed $z\neq i$. Therefore the angle which the segment from $-i$ to $z$ makes with the real axis is the same as the angle which the segment from $-i$ to $z'$ makes, and we must have $z=z'$.
      
      Thus no two of the fiberwise homogeneous fibrations by geodesics in the statement of the theorem are equivalent.
    \end{proof}

\bibliography{thesis}{}

\providecommand{\bysame}{\leavevmode\hbox to3em{\hrulefill}\thinspace}
\providecommand{\MR}{\relax\ifhmode\unskip\space\fi MR }
\providecommand{\MRhref}[2]{%
  \href{http://www.ams.org/mathscinet-getitem?mr=#1}{#2}
}
\providecommand{\href}[2]{#2}
\begin{thebibliography}{1}

\bibitem{beckers2008subgroups}
J~Beckers, J~Patera, M~Perroud, and P~Winternitz, \emph{Subgroups of the
  {E}uclidean group and symmetry breaking in nonrelativistic quantum
  mechanics}, Journal of mathematical physics \textbf{18} (2008), no.~1,
  72--83.

\bibitem{godoysalvai2014}
Yamile Godoy and Marcos Salvai, \emph{Global smooth geodesic foliations of the
  hyperbolic space}, \href{http://arxiv.org/abs/1411.5701}{arXiv:1411.5701}
  (2014).

\bibitem{gromoll1988low}
Detlef Gromoll and Karsten Grove, \emph{The low-dimensional metric foliations
  of {E}uclidean spheres}, J. Differential Geom \textbf{28} (1988), no.~1,
  143--156.

\bibitem{nuchi2014hopf}
Haggai Nuchi, \emph{{H}opf fibrations are characterized by being fiberwise
  homogeneous}, \href{http://arxiv.org/abs/1407.4549}{arXiv:1407.4549} (2014).

\bibitem{nuchi2014surprising}
\bysame, \emph{A surprising fibration of {$S^3 \times S^3$} by great
  3-spheres}, \href{http://arxiv.org/abs/1407.4548}{arXiv:1407.4548} (2014).

\bibitem{salvai2009global}
Marcos Salvai, \emph{Global smooth fibrations of $\mathbb{R}^3$ by oriented
  lines}, Bulletin of the London Mathematical Society \textbf{41} (2009),
  no.~1, 155--163.

\bibitem{shaw1970subgroup}
Ronald Shaw, \emph{The subgroup structure of the homogeneous {L}orentz group},
  The Quarterly Journal of Mathematics \textbf{21} (1970), no.~1, 101--124.

\bibitem{thurston1997three}
William~P Thurston and Silvio Levy, \emph{Three-dimensional geometry and
  topology}, vol.~1, Princeton university press, 1997.

\bibitem{wilking2001index}
Burkhard Wilking, \emph{Index parity of closed geodesics and rigidity of {H}opf
  fibrations}, Inventiones mathematicae \textbf{144} (2001), no.~2, 281--295.

\end{thebibliography}
\bibliographystyle{amsplain}
\end{document}